\documentclass[reqno]{amsart}
\usepackage{amsmath}
\usepackage{amssymb}
\usepackage{amsthm,amsxtra}
\usepackage{cite}
\usepackage{enumerate}
\usepackage{enumitem}
\usepackage[mathscr]{eucal}
\usepackage{float}
\usepackage{centernot}
\usepackage{tikz-cd}
\usepackage{tkz-graph}
\usepackage{graphicx}
\usepackage{mathtools}
\usepackage{lineno}
\usetikzlibrary{positioning}
\usepackage{adjustbox}

\newtheorem{Def}{Definition}[section]
\newtheorem{Th}{Theorem}[section]
\newtheorem{Ex}{Example}[section]

\newtheorem{Lemma}{Lemma}[section]
\newtheorem{Prop}{Proposition}[section]
\newtheorem{Cor}{Corollary}[section]
\newtheorem{Rem}{Remark}[section]

{}

\DeclareMathOperator{\cov}{\sf cov}

\DeclareMathOperator{\Int}{Int}
\newcommand{\Hu}{{\sf H}}
\newcommand{\Me}{{\sf M}}
\newcommand{\Ro}{{\sf R}}

\newcommand{\Hub}{{\sf H}\text{-bounded}}
\newcommand{\Meb}{{\sf M}\text{-bounded}}
\newcommand{\Rob}{{\sf R}\text{-bounded}}
\newcommand{\Ub}{\Upsilon\text{-bounded}}
\newcommand{\U}{\Upsilon}
\begin{document}

\title[On certain localized version of uniform selection principles]{On certain localized version of uniform selection principles}

\author[ N. Alam, D. Chandra ]{ Nur Alam$^*$, Debraj Chandra$^*$ }
\newcommand{\acr}{\newline\indent}
\address{\llap{*\,}Department of Mathematics, University of Gour Banga, Malda-732103, West Bengal, India}
\email{nurrejwana@gmail.com, debrajchandra1986@gmail.com}

\thanks{ The first author
is thankful to University Grants Commission (UGC), New Delhi-110002, India for granting UGC-NET Junior Research Fellowship (1173/(CSIR-UGC NET JUNE 2017)) during the tenure of which this work was done.}


\maketitle
\begin{abstract}
We intend to localize the selection principles in uniform spaces (Ko\v{c}inac, 2003) by introducing their local variations, namely locally $\Upsilon$-bounded spaces (where $\Upsilon$ is Menger, Hurewicz or Rothberger). It has been observed that the difference between uniform selection principles and the corresponding local correlatives as introduced here is reasonable enough to discuss about these new notions. Certain observations using the critical cardinals (on the uniform selection principles which have not studied before) as well as preservation like properties (on the local versions) are presented. The interrelationships between the notions considered in this paper are outlined into an implication diagram. Certain interactions between these local variations are also investigated. We present several examples to illustrate the distinguishable behaviour of the new notions.
\end{abstract}
\smallskip

\noindent{\bf\keywordsname{}:} {Uniform space, selection principles, ${\sf M}$-bounded, ${\sf H}$-bounded, ${\sf R}$-bounded, locally ${\sf M}$-bounded, locally ${\sf H}$-bounded, locally ${\sf R}$-bounded, locally precompact, locally pre-Lindel\"{o}f.}

\section{Introduction}
There is a long illustrious history of study of selection principles in set-theoretic topology. This vast field in topology became more popular and had attracted a lot of researcher's attention in the last twenty five years after Scheepers' seminal paper \cite{coc1} (see also \cite{coc2}), where a systematic study in this fascinating field was initiated. Since then, various topological notions have been defined or characterized in terms of the classical selection principles. Interested readers may explore the survey papers \cite{SRSP,SCTU,Survey} for more information on this topic.

In 2003, Ko\v{c}inac \cite{SPUS} introduced the study of selection principles in uniform spaces by defining uniform analogues of Menger, Hurewicz and Rothberger covering properties, namely $\Meb$, $\Hub$ and $\Rob$ respectively and differentiated these uniform variations from classical Menger, Hurewicz and Rothberger properties. Interestingly it was observed that these uniform covering properties can also be defined in terms of star selection principles in unform spaces.

Later in 2013, Ko\v{c}inac and K\"{u}nzi \cite{SPUR} further extended the study to quasi-uniform spaces. For more information about the uniform selection principles, we refer the reader to consult the papers \cite{SRSP,SCTU,Maio,kocqm,BPFS,UBFS} and references therein.

This paper is a continuation of the study of uniform selection principles started in \cite{SPUS} and is organised as follows. In Section 3, we present certain observations on uniform selection principles (that were not at all investigated earlier in uniform structures), which seem to be effective in our context.
In Section 4, we make an effort to extend the concept of uniform selection principles by introducing local variations of these selection principles, namely locally $\Meb$, locally $\Hub$ and locally $\Rob$ spaces (for similar type of investigations, see \cite{dcna21}). Certain situations are described which witness that these local variations behave much differently from the uniform selection principles. Later in this section, preservation like properties of the new notions are investigated carefully and the interrelationships between these new notions are also discussed.
Section 5 is the final portion of this article, which is devoted to present illustrative examples. It is shown that the class consisting of each of these local variations is strictly larger than the class containing the corresponding uniform counterparts. We also present exemplary observations of their perceptible behaviours.

\section{Preliminaries}
For undefined notions and terminologies, see \cite{Engelking}. We start with some basic information about uniform spaces.

Let $X$ be a set and let $A,B\subseteq X\times X$. We define $A^{-1}=\{(x,y) : (y,x)\in A\}$ and $A\circ B=\{(x,y) : \exists \,z\in X \:\text{such that}\: (x,z)\in A \;\text{and}\; (z,y)\in B\}$. The diagonal of $X\times X$ is the set $\Delta=\{(x,x) : x\in X\}$. A set $U\subseteq X\times X$ is said to be an entourage of the diagonal if $\Delta\subseteq U$ and $U^{-1}=U$. The family of all entourages of the diagonal $\Delta\subseteq X\times X$ will be denoted by $E_X(\Delta)$. If $F\subseteq X$ and $U\in E_X(\Delta)$, then $U[F]=\cup_{x\in F}U[x]$, where $U[x]=\{y\in X : (x,y)\in U\}$. Recall that a uniform space can be described equivalently in terms of either a diagonal uniformity or a covering uniformity \cite{Engelking} (see also \cite{Tukey,Borubaev}). In this paper we use diagonal uniformity to define a uniform space.
 A uniformity on a set $X$ is a subfamily $\mathbb{U}$ of $E_X(\Delta)$ which satisfies the following conditions.
 (i) If $U\in\mathbb{U}$ and $V\in E_X(\Delta)$ with $U\subseteq V$, then $V\in\mathbb{U}$; (ii) If $U,V\in\mathbb{U}$, then $U\cap V\in\mathbb{U}$; (iii) For every $U\in\mathbb{U}$, there exists a $V\in\mathbb{U}$ such that $V\circ V\subseteq U$; and (iv) $\cap\mathbb{U}=\Delta$. The pair $(X,\mathbb{U})$ is called a uniform space \cite{Engelking}.
Clearly, every uniform space $(X,\mathbb{U})$ is a topological space. The family $\tau_{\mathbb{U}}=\{O\subseteq X : \:\text{for each} \: x\in O\:\text{there exists a}\: U\in\mathbb{U}\:\text{such that}\: U[x]\subseteq O\}$ is the topology on $X$ generated by the uniformity $\mathbb{U}$. It is well known that the topology of a space $X$ can be induced by a uniformity on $X$ if and only if $X$ is Tychonoff (see \cite[Theorem 8.1.20]{Engelking}). If $(X,d)$ is a metric space, then the family $\{U_\varepsilon : \varepsilon>0\}$, where $U_\varepsilon=\{(x,y)\in X\times X : d(x,y)<\varepsilon\}$, is a base for the uniformity $\mathbb{U}$ induced by the metric $d$. Moreover, the topologies induced on $X$ by the uniformity $\mathbb{U}$ and by the metric $d$ coincide.
By a subspace $Y$ of a uniform space $(X,\mathbb{U})$ we mean the uniform space $(Y,\mathbb{U}_Y)$, where $Y\subseteq X$ and $\mathbb{U}_Y=\{(Y\times Y)\cap U : U\in\mathbb{U}\}$ (which is called the relative uniformity on $Y$).
Let $\mathcal{F}$ be a family of subsets of $X$. We say that $\mathcal{F}$ contains arbitrarily small sets if for every $U\in\mathbb{U}$ there exists a $F\in\mathcal{F}$ such that $F\times F\subseteq U$ (see \cite{Engelking}).
We say that $X$ is complete if every family $\mathcal{F}$ of closed subsets of $X$ which has the finite intersection property and contains arbitrarily small sets has nonempty intersection. A uniformity $\mathbb{U}$ on a set $X$ is complete if the space $(X,\mathbb{U})$ is complete \cite{Engelking}. A function $f:(X,\mathbb{U})\to(Y,\mathbb{V})$ between two uniform spaces is uniformly continuous if for every $V\in\mathbb{V}$ there exists a $U\in\mathbb{U}$ such that for all $x,y\in X$ we have $(f(x),f(y))\in V$ whenever $(x,y)\in U$. A bijective mapping $f:(X,\mathbb{U})\to(Y,\mathbb{V})$ is said to be a uniform isomorphism if both $f$ and $f^{-1}$ are uniformly continuous (see \cite{Engelking}). We say that two uniform spaces $X$ and $Y$ are uniformly isomorphic if there exists a uniform isomorphism of $X$ onto $Y$. It is clear that every uniform isomorphism is an open mapping.
$(X,\mathbb{U})$ is said to be precompact or totally bounded (resp. pre-Lindel\"{o}f) if for each $U\in\mathbb{U}$ there exists a finite (resp. countable) $A\subseteq X$ such that $U[A]=X$ \cite{Engelking,Borubaev,preLindelof}.
Every compact (resp. Lindel\"{o}f) uniform space is precompact (resp. pre-Lindel\"{o}f). Moreover, for a complete uniform space precompactness (resp. pre-Lindel\"{o}fness) and compactness (resp. Lindel\"{o}fness) are equivalent. It is easy to observe that if the uniformity $\mathbb{U}$ on a set $X$ is induced by a metric $d$, then $(X,\mathbb{U})$ is complete (resp. precompact, pre-Lindel\"{o}f) if and only if $(X,d)$ is complete (resp. precompact, pre-Lindel\"{o}f).

We now recall some definitions of topological spaces formulated in terms of classical selection principles from \cite{coc1,coc2}.
A topological space $X$ is said to be Menger (resp. Rothberger) if for each sequence $(\mathcal{U}_n)$ of open covers of $X$ there is a sequence $(\mathcal{V}_n)$ (resp. $(U_n)$) such that for each $n$ $\mathcal{V}_n$ is a finite subset of $\mathcal{U}_n$ (resp. $U_n\in\mathcal{U}_n)$ and $\cup_{n\in\mathbb{N}}\mathcal{V}_n$ (resp. $\{U_n : n\in\mathbb{N}\}$) is an open cover of $X$. A topological space $X$ is said to be Hurewicz if for each sequence $(\mathcal{U}_n)$ of open covers of $X$ there is a sequence $(\mathcal{V}_n)$ such that for each $n$ $\mathcal{V}_n$ is a finite subset of $\mathcal{U}_n$  and each $x\in X$ belongs to $\cup\mathcal{V}_n$ for all but finitely many $n$. A topological space $X$ is said to be locally compact (resp. locally Menger, locally Hurewicz, locally Rothberger, locally Lindel\"{o}f) if for each $x\in X$ there exist an open set $U$ and a compact (resp. Menger, Hurewicz, Rothberger, Lindel\"{o}f) subspace $Y$ of $X$ such that $x\in U\subseteq Y$.

In \cite{SPUS} (see also \cite{SRSP}), Ko\v{c}inac introduced the following uniform selection principles. A uniform space $(X,\mathbb{U})$ is  Menger-bounded (in short, $\Meb$) if for each sequence $(U_n)$ of members of $\mathbb{U}$ there is a sequence $(F_n)$ of finite subsets of $X$ such that $\cup_{n\in\mathbb{N}}U_n[F_n]=X$. A uniform space $(X,\mathbb{U})$ is said to be Hurewicz-bounded (in short, $\Hub$) if for each sequence $(U_n)$ of members of $\mathbb{U}$ there is a sequence $(F_n)$ of finite subsets of $X$ such that each $x\in X$ belongs to $U_n[F_n]$ for all but finitely many $n$.
Also a uniform space $(X,\mathbb{U})$ is said to be Rothberger-bounded (in short, $\Rob$) if for each sequence $(U_n)$ of members of $\mathbb{U}$ there is a sequence $(x_n)$ of members of $X$ such that $\cup_{n\in\mathbb{N}}U_n[x_n]=X$.
The above properties are also known as uniformly Menger, uniformly Hurewicz and uniformly Rothberger respectively.
We also say that a metric space $(X,d)$ is $\Me$-bounded (resp. $\Hu$-bounded, $\Ro$-bounded) if $X$ with the induced uniformity is $\Me$-bounded (resp. $\Hu$-bounded, $\Ro$-bounded).

Throughout the paper $(X,\mathbb{U})$ (or $X$ for short, when $\mathbb U$ is clear from the context) stands for a uniform space, where $\mathbb U$ is a diagonal uniformity on $X$.\\

The following two results are useful in our context.
\begin{Th}[cf. \cite{SPUS}]
Let $(X,\mathbb{U})$ and $(Y,\mathbb{V})$ be two uniform spaces.
\begin{enumerate}[wide=0pt,label={\upshape(\arabic*)},ref={\theTh(\arabic*)},leftmargin=*]
  \item \label{LU501}  $(X\times Y,\mathbb{U}\times\mathbb{V})$ is $\Hu$-bounded if and only if both $(X,\mathbb{U})$ and $(Y,\mathbb{V})$ are $\Hu$-bounded.
  \item 
  If $(X,\mathbb{U})$ is $\Me$-bounded and $(Y,\mathbb{V})$ is precompact, then $(X\times Y,\mathbb{U}\times\mathbb{V})$ is $\Me$-bounded.
\end{enumerate}
\end{Th}

\begin{Th}[cf. \cite{SPUS}]
\hfill
\begin{enumerate}[wide=0pt,label={\upshape(\arabic*)},ref={\theTh(\arabic*)},leftmargin=*]
  \item \label{LU404} $\Meb$, $\Hub$ and $\Rob$ properties are hereditary and preserved under uniformly continuous mappings.
  \item \label{LU401} Let $(X,\mathbb{U})$ be a uniform space and $Y\subseteq X$. If $(Y,\mathbb{U}_Y)$ is $\Hu$-bounded, then $(\overline{Y},\mathbb{U}_{\overline{Y}})$ is also $\Hu$-bounded.
  \item \label{LU402} If $(X,\mathbb{U})$ is a complete uniform space, then $(X,\mathbb{U})$ is Hurewicz, if and only if it is $\Hu$-bounded.

  \item \label{LU403} If $(X,\mathbb{U})$ is  Menger, Hurewicz, Rothberger, then it is also $\Me$-bounded, $\Hu$-bounded and $\Ro$-bounded respectively.
\end{enumerate}
\end{Th}

\section{Few observations on uniform selection principles}

We present a few more observations on uniform selection principles that will be useful subsequently. Throughout the paper we use the symbol $\U$ to denote  any of the {Menger}, {Hurewicz} or {Rothberger} properties. Accordingly $\Ub$ denotes any of the $\Meb$, $\Hub$ or $\Rob$ properties.

We start with two basic observations (without proof) about uniform selection principles.
\begin{Lemma}
Let $(X,\mathbb{U})$ be a uniform space. A subspace $Y$ of $X$ is
\begin{enumerate}[wide=0pt,label={\upshape(\arabic*)},
ref={\theLemma(\arabic*)},leftmargin=*]
  \item 
  $\Me$-bounded  if and only if for each sequence $(U_n)$ of members of $\mathbb{U}$ there exists a sequence $(F_n)$ of finite subsets of $Y$ such that $Y\subseteq\cup_{n\in\mathbb{N}}U_n[F_n]$.
  \item \label{LU603} $\Hu$-bounded if and only if for each sequence $(U_n)$ of members of $\mathbb{U}$ there exists a sequence $(F_n)$ of finite subsets of $Y$ such that each $y\in Y$ belongs to $U_n[F_n]$ for all but finitely many $n$.
  \item 
  $\Ro$-bounded if and only if for each sequence $(U_n)$ of members of $\mathbb{U}$ there exists a sequence $(x_n)$ of members of $Y$ such that $Y\subseteq\cup_{n\in\mathbb{N}}U_n[x_n]$.
\end{enumerate}
\end{Lemma}

\begin{Lemma}
\label{TU2}
Let $\mathbb{U}$ and $\mathbb{V}$ be two uniformities on a set $X$ such that $\mathbb{V}$ is finer than $\mathbb{U}$. If $(X,\mathbb{V})$ is $\Ub$, then $(X,\mathbb{U})$ is also $\Ub$.
\end{Lemma}

\begin{Th}
\label{TU4}
Let $(X,\mathbb{U})$ be a uniform space with $X=\cup_{n\in\mathbb{N}}X_n$. Then $X$ is $\Ub$ if and only if  $X_n$ is $\Ub$ for each $n$.
\end{Th}
\begin{proof}
We only prove sufficiency for the case of $\Hu$-bounded.
Let $(U_n)$ be a sequence of members of $\mathbb{U}$. By Lemma~\ref{LU603}, for each $k\in\mathbb{N}$ we can choose a sequence $(F_n^{(k)}:{n\geq k})$ of finite subsets of $X_k$ such that each $x\in X_k$ belongs to $U_n[F_n^{(k)}]$ for all but finitely many $n\geq k$. For each $n$ let $F_n=\cup_{k\leq n}F_n^{(k)}$. Then $(F_n)$ is a sequence of finite subsets of $X$. We show that each $x\in X$ belongs to $U_n[F_n]$ for all but finitely many $n$. Let $x\in X$. Choose $k_0\in\mathbb{N}$ such that $x\in X_{k_0}$. Clearly, $x\in U_n[F_n^{(k_0)}]$ for all but finitely many $n\geq k_0$ and hence $x\in U_n[F_n]$ for all but finitely many $n$ since $F_n^{(k_0)}\subseteq F_n$ for all $n\geq k_0$. Hence the result.
\end{proof}

Let $(Y,\mathbb{V})$ be a uniform space and $X$ be a set. If $f:X\to Y$ is an injective mapping, then there is a natural uniformity on $X$ induced by $f$ and denoted by $f^{-1}(\mathbb{V})$. This uniformity is generated by the base $\{g^{-1}(V) : V\in\mathbb{V}\}$, where $g:X\times X\to Y\times Y$ is defined by $g(x,y)=(f(x),f(y))$; that is $g=f\times f$. 

\begin{Th}
\label{L4}
 Let $f:X\to Y$ be an injective mapping from a set $X$ onto a uniform space $(Y,\mathbb{V})$. If $Y$ is $\Ub$, then $X$ is also $\Ub$.
\end{Th}
\begin{proof}
We give a proof for the case of $\Hu$-bounded.
Suppose that $Y$ is $\Hu$-bounded.
Let $(U_n)$ be a sequence of members of $f^{-1}(\mathbb{V})$. By definition, $\{g^{-1}(V) : V\in\mathbb{V}\}$ is a base for the uniformity $f^{-1}(\mathbb{V})$, where $g=f\times f$. For each $n$ choose $V_n\in\mathbb{V}$ such that $g^{-1}(V_n)\subseteq U_n$.
Next choose a sequence $(F_n)$ of finite subsets of $Y$ such that for each $y\in Y$ there exists a $n_y\in\mathbb{N}$ such that $y\in V_n[F_n]$ for all $n\geq n_y$. For each $n$ set $F_n=\{y_1^{(n)},y_2^{(n)},\cdots,y_{k_n}^{(n)}\}$ and $F_n^\prime=\{x_1^{(n)},x_2^{(n)},\cdots,x_{k_n}^{(n)}\}
\subseteq X$, where $f(x_i^{(n)})=y_i^{(n)}$ for each $1\leq i\leq k_n$. 
Let $x\in X$. Then there exists a $n_{f(x)}\in\mathbb{N}$ such that $f(x)\in V_n[F_n]$ for all $n\geq n_{f(x)}$. For each $n\geq n_{f(x)}$ there exists a $i_n$ with $1\leq i_n\leq k_n$ such that $f(x)\in V_n[y_{i_n}^{(n)}]$. Thus, for each $n\geq n_{f(x)}$ we have $g(x,x_{i_n}^{(n)})\in V_n$; that is $x\in U_n[x_{i_n}^{(n)}]\subseteq U_n[F_n^\prime]$ for all but finitely many $n$. This completes the proof.
\end{proof}

 The eventual dominance relation $\leq^*$ on the Baire space $\mathbb{N}^\mathbb{N}$ is defined by $f\leq^*g$ if and only if $f(n)\leq g(n)$ for all but finitely many $n$. A subset $A$ of $\mathbb{N}^\mathbb{N}$ is said to be dominating if for each $g\in\mathbb{N}^\mathbb{N}$ there exists a $f\in A$ such that $g\leq^* f$. A subset $A$ of $\mathbb{N}^\mathbb{N}$ is said to be bounded if there is a $g\in\mathbb{N}^\mathbb{N}$ such that $f\leq^*g$ for all $f\in A$. Moreover a set $A\subseteq\mathbb{N}^\mathbb{N}$ is said to be guessed by $g\in\mathbb{N}^\mathbb{N}$ if $\{n\in\mathbb{N} : f(n)=g(n)\}$ is infinite for all $f\in A$. The minimum cardinality of a dominating subset of $\mathbb{N}^\mathbb{N}$ is denoted by $\mathfrak{d}$, and the minimum cardinality of a unbounded subset of $\mathbb{N}^\mathbb{N}$ is denoted by $\mathfrak{b}$. Let $\cov(\mathcal{M})$ be the minimum cardinality of a family of meager subsets of $\mathbb{R}$ that covers $\mathbb{R}$. In \cite{CAMC} (see also \cite[Theorem 2.4.1]{TBHJ}), $\cov(\mathcal{M})$ is described as the minimum cardinality of a subset $F\subseteq\mathbb{N}^\mathbb{N}$ such that for every $g\in\mathbb{N}^\mathbb{N}$ there is $f\in F$ such that $f(n)\neq g(n)$ for all but finitely many $n$. Thus, we can say that if $F\subseteq\mathbb{N}^\mathbb{N}$ and $|F|<\cov(\mathcal{M})$, then $F$ can be guessed by a $g\in\mathbb{N}^\mathbb{N}$. It is to be noted that the Baire space is also a uniform space with the uniformity $\mathbb B$ induced by the Baire metric.

\begin{Th}
\label{TN01}
Every pre-Lindel\"{o}f space $(X,\mathbb{U})$ with $|X|<\mathfrak d$  is $\Meb$.
\end{Th}
\begin{proof}
Let $(U_n)$ be a sequence of members of $\mathbb{U}$. Apply the pre-Lindel\"{o}f property to obtain a sequence $(A_n)$ of countable subsets of $X$ such that $U_n[A_n]=X$ for each $n$. Say $A_n=\{x_m^{(n)} : m\in\mathbb{N}\}$ for each $n$. Now for each $x\in X$ define  $f_x\in\mathbb{N}^\mathbb{N}$ by $f_x(n)=\min\{m\in\mathbb{N} : x\in U_n[x_m^{(n)}]\}$, $n\in\mathbb{N}$. Since the cardinality of $\{f_x : x\in X\}$ is less than $\mathfrak{d}$, there is a $g\in\mathbb{N}^\mathbb{N}$ and for $x\in X$ a $n_x\in\mathbb{N}$ such that $f_x(n_x)<g(n_x)$. For each $n$ define $F_n=\{x_m^{(n)} : m\leq g(n)\}$. Observe that if $x\in X$, then $x\in U_{n_x}[F_{n_x}]$. Clearly, $\{U_n[F_n] : n\in\mathbb{N}\}$ covers $X$ and hence $X$ is $\Meb$.
\end{proof}

Similarly we obtain the following.
\begin{Th}
\label{TN02}
Every pre-Lindel\"{o}f space $(X,\mathbb{U})$ with $|X|<\mathfrak{b}$ is $\Hub$.
\end{Th}

\begin{Th}
\label{TN03}
Every pre-Lindel\"{o}f space $(X,\mathbb{U})$ with $|X|<\cov(\mathcal{M})$ is $\Rob$.
\end{Th}
\begin{proof}
Let $(U_n)$ be a sequence of members of $\mathbb{U}$. Consider $A_n$ and $f_x$ as in Theorem~\ref{TN01} and proceed as follows.

Since the cardinality of $\{f_x : x\in X\}$ is less than $\cov(\mathcal{M})$, choose $g\in\mathbb{N}^\mathbb{N}$ such that $\{n : f_x(n)=g(n)\}$ is infinite for all $x\in X$.

Observe that if $x\in X$, then $f_x(n_x)=g(n_x)$ for some positive integer $n_x$; that is $x\in U_{n_x}[x_{g(n_x)}^{(n_x)}]$. Thus, $\{U_n[x_{g(n)}^{(n)}] : n\in\mathbb{N}\}$ is a cover of $X$, so that $X$ is $\Rob$.
\end{proof}

\begin{Th}
\label{TU8}
If $(X,\mathbb{U})$ is $\Me$-bounded, then any uniformly continuous image of $X$ into $\mathbb{N}^\mathbb{N}$ is non-dominating.
\end{Th}
\begin{proof}
In view of Theorem~\ref{LU404}, we can assume that $X$ is an $\Meb$ subspace of $\mathbb{N}^\mathbb{N}$. For each $n\in\mathbb{N}$ consider $U_n=\{(\varphi,\psi)\in\mathbb{N}^\mathbb{N}\times\mathbb{N}^\mathbb{N} : \varphi(n)=\psi(n)\}\in\mathbb{B}$. Let $\{P_n : n\in\mathbb{N}\}$ be a partition of $\mathbb{N}$ into pairwise disjoint infinite subsets.  For each $n\in\mathbb{N}$, apply the $\Meb$ property of $X$ to $(U_k : k\in\ P_n)$ to obtain a sequence $(F_k : k\in P_n)$ of finite subsets of $X$ such that $X\subseteq\cup_{k\in P_n} U_k[F_k]$. Thus, $F_n$ is defined for each $n\in \mathbb N$. Now define $g:\mathbb{N}\to\mathbb{N}$ by $g(n)=1+\max\{f(n) : f\in F_n\}$. To complete the proof we show that $X$ is not dominating (which is witnessed by $g$).

Let $f\in X$. For each $k\in\mathbb{N}$ choose $n_k\in P_k$ such that $f\in U_{n_k}[F_{n_k}]$. Again choose for each $k\in\mathbb{N}$ a $f_k\in F_{n_k}$ such that $(f,f_k)\in U_{n_k}$; that is $f(n_k)=f_k(n_k)$ for all $k\in\mathbb{N}$. Consequently, $f(n_k)<g(n_k)$ for all $k\in\mathbb{N}$ and hence the set $\{n\in\mathbb{N} : g(n)\nleq f(n)\}$ is infinite.
\end{proof}

\begin{Th}
\label{TU9}
If $(X,\mathbb{U})$ is $\Hu$-bounded, then any uniformly continuous image of $X$ into $\mathbb{N}^\mathbb{N}$ is bounded (with respect to $\leq^*$).
\end{Th}
\begin{proof}
The proof is modelled in the proof of Theorem~\ref{TU8}. It remains to observe that the respective map $g$ dominates every element of $X$.
\end{proof}

\begin{Th}
\label{TU10}
If $(X,\mathbb{U})$ is $\Ro$-bounded, then any uniformly continuous image of $X$ into $\mathbb{N}^\mathbb{N}$ can be guessed.
\end{Th}
\begin{proof}
We closely follow the proof of Theorem~\ref{TU8}.
Choose $U_n$'s as in Theorem~\ref{TU8} and proceed with the following modifications. For each $n\in\mathbb{N}$ choose a sequence $(f_k : k\in P_n)$ of members of $X$ such that $X\subseteq\cup_{k\in P_n}U_k[f_k]$. Thus, $f_n$ is defined for each positive integer $n$. Define $g:\mathbb{N}\to\mathbb{N}$ by $g(n)=f_n(n)$. We now show that $X$ is guessed by $g$. Choose any $f\in X$ and for each $k\in\mathbb{N}$ choose $n_k\in P_k$ such that $(f,f_{n_k})\in U_{n_k}$. Clearly, the set $\{n\in\mathbb N: f(n)=g(n)\}$ is infinite (since it contains all $n_k$'s) and this completes the proof.
\end{proof}

\section{Local variations of uniform selection principles}
\subsection{Locally $\Ub$ spaces}
We now introduce the main definition of this paper.

\begin{Def}
 Let $(X,\mathbb{U})$ be a uniform space. Then $X$ is said to be locally $\Ub$ if for each $x\in X$ there exists a $U\in\mathbb{U}$ such that $U[x]$ is a $\Ub$ subspace of $X$.
\end{Def}
\begin{Rem}
Locally precompact and locally pre-Lindel\"{o}f spaces can be similarly defined.
\end{Rem}

From the above definition and \cite{SPUS}, we obtain the following implication diagram (where the abbreviations ${\sf C, \;H, \;L, \;M}$, ${\sf R}$ denote respectively compact, Hurewicz, Lindel\"{o}f, Menger and Rothberger spaces, the prefixes $l$, ${\sf p}$ stand respectively for `locally' and `pre-' and the suffix ${\sf b}$ stands for `-bounded').

\begin{figure}[h]
\begin{adjustbox}{max width=\textwidth,max height=\textheight,keepaspectratio,center}
\begin{tikzcd}[column sep=4ex,row sep=4ex,arrows={crossing over}]
&& {\sf L} \arrow[rr]&&
{\sf pL}  \arrow[rr]&&
l{\sf pL}
\\
& {\sf M} \arrow[ur]\arrow[rr]&&
{\sf Mb}\arrow[ur]\arrow[rr]&&
l{\sf Mb}\arrow[ur]&
\\
{\sf H} \arrow[ur]\arrow[rr]&&
{\sf Hb} \arrow[ur]\arrow[rr]&&
l{\sf Hb} \arrow[ur]&&
\\
&{\sf R}\arrow[uu]\arrow[rr]&&
{\sf Rb} \arrow[uu]\arrow[rr]&&
l{\sf Rb}\arrow[uu]&
\\
{\sf C} \arrow[uu]\arrow[rr]&&
{\sf pC} \arrow[rr]
\arrow[uu]&&
l{\sf pC}\arrow[uu]&&
\end{tikzcd}
\end{adjustbox}
 \caption{Diagram of local properties in uniform spaces}
 \label{dig1}
\end{figure}
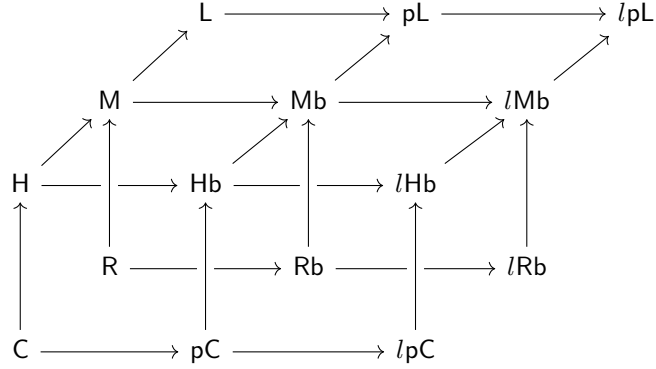

\newpage
We now present equivalent formulations of the new notions.
\begin{Th}
If $(X,\mathbb{U})$ is a uniform space, then the following assertions are equivalent.
\begin{enumerate}[label={\upshape(\arabic*)}]
  \item $X$ is locally $\Ub$.
  \item For each $x\in X$ and $V\in\mathbb{U}$ there exist a $U\in\mathbb{U}$ and a $\Ub$ subspace $Y$ of $X$ such that $U[x]\subseteq Y\subseteq V[x]$.
  \item For each $x\in X$ there exist a $U\in\mathbb{U}$ and a $\Ub$ subspace $Y$ of $X$ such that $U[x]\subseteq Y$.
  \item For each $x\in X$ there exists a $U\in\mathbb{U}$ such that $\overline{U[x]}$ is a $\Ub$ subspace of $X$.
\end{enumerate}
\end{Th}
\begin{proof}
We only present proof of $(1)\Rightarrow (2)$. Let $x\in X$ and $V\in\mathbb{U}$. We can find a $U\in\mathbb{U}$ such that $U[x]$ is a $\Ub$ subspace of $X$. Clearly, $U\cap V\in\mathbb{U}$ and $Y=U[x]\cap V[x]$ is a $\Ub$ subspace of $X$ with $(U\cap V)[x]\subseteq Y\subseteq V[x]$. Hence $(2)$ holds.
\end{proof}

 We say that a metric space $(X,d)$ is locally $\Ub$ if for each $x\in X$ there exists an open set $V$ in $X$ such that $x\in V$ and $V$ is a $\Ub$ subspace of $(X,d)$. Likewise locally precompact and locally pre-Lindel\"{o}f metric spaces can also be defined.

\begin{Rem}
It is easy to observe that if the uniformity $\mathbb{U}$ on a set $X$ is induced by a metric $d$, then the uniform space $(X,\mathbb{U})$ is locally $\Ub$ if and only if the metric space $(X,d)$ is locally $\Ub$. Similar assertion holds for locally precompact and locally pre-Lindel\"{o}f metric spaces.
\end{Rem}

Using Theorem~\ref{TU4}, we have the following observation.

\begin{Prop}
\label{PU3}
Let $(X,\mathbb{U})$ be a uniform space such that $X$ is Lindel\"{o}f. Then $X$ is  $\Ub$ if and only if $X$ is locally $\Ub$.
\end{Prop}

\begin{Prop}
\label{PU2}
Let $(X,\mathbb{U})$ be a uniform space. If $X$ is locally $\U$, then $X$ is locally $\Ub$.
\end{Prop}
\begin{proof}
Let $x\in X$. Choose an open set $U$ and a $\U$ subspace $Y$ of $X$ such that $x\in U\subseteq Y$. By Theorem~\ref{LU403}, $Y$ is a $\Ub$ subspace of $X$. Since $x\in U$ is open in $X$, choose $V\in\mathbb{U}$ such that $V[x]\subseteq U$. We thus obtain a $\Ub$ subspace $V[x]$ of $X$, which shows that $X$ is locally $\Ub$.
\end{proof}

It can also be observed that locally compact (resp. locally Lindel\"{o}f) implies locally precompact (resp. locally pre-Lindel\"{o}f).

\begin{Th}
\label{P41}
If $(X,\mathbb{U})$ is a complete uniform space, then $X$ is locally Hurewicz if and only if $X$ is locally $\Hu$-bounded.
\end{Th}
\begin{proof}
The necessity follows from Proposition~\ref{PU2}.

Conversely assume that $X$ is locally $\Hu$-bounded. Let $x\in X$. Choose $U\in\mathbb{U}$ such that $U[x]$ is  $\Hu$-bounded. Clearly, $x\in\Int U[x]\subseteq\overline{U[x]}$. By Theorem~\ref{LU401}, $\overline{U[x]}$ is a $\Hu$-bounded subspace of $X$. Again by Theorem~\ref{LU402}, $\overline{U[x]}$ is a Hurewicz subspace of $X$ as $\overline{U[x]}$ is complete. Thus, $X$ is locally Hurewicz.
\end{proof}

\begin{Th}
Let $(X,\mathbb{U})$ be locally $\Ub$. An element $U\in\mathbb{U}$ is open in $X\times X$ if and only if $(Y\times Z)\cap U$ is open in $Y\times Z$ for every two $\Ub$ subspaces $Y$ and $Z$ of $X$.
\end{Th}
\begin{proof}

We only need to prove sufficiency. Let $U\in\mathbb{U}$ and $(x,y)\in U$. Since $X$ is locally $\Ub$, there exist $V,W\in\mathbb{U}$ such that $V[x]$ and $W[y]$ are $\Ub$ subspaces of $X$. Clearly, $(\Int V[x]\times \Int W[y])\cap U$ is open in $\Int V[x]\times\Int W[y]$ since $(V[x]\times W[y])\cap U$ is open in $V[x]\times W[y]$. Evidently $(x,y)\in(\Int V[x]\times\Int W[y])\cap U$ is open in $X\times X$ and hence $U$ is open in $X\times X$ as required.
\end{proof}

The following observation is due to Theorems~\ref{TN01}, \ref{TN02} and \ref{TN03} with necessary modifications.
\begin{Prop}
Every locally pre-Lindel\"{o}f space with cardinality less than $\mathfrak{d}$ (resp. $\mathfrak{b}$, $\cov(\mathcal{M})$) is locally $\Meb$ (resp. locally $\Hub$, locally $\Rob$).
\end{Prop}

\subsection{Some observations on locally $\Ub$ spaces}

We now present some preservation like properties of these local variations under certain topological operations.
\begin{Prop}
\label{P4}
Locally $\Ub$ property is hereditary.
\end{Prop}

\begin{Prop}
 Let $f:X\to Y$ be an injective mapping from a set $X$ onto a uniform space $(Y,\mathbb{V})$. If $Y$ is locally $\Ub$, then $X$ is also locally $\Ub$.
\end{Prop}

\begin{proof}
First recall that the uniformity $f^{-1}(\mathbb{V})$ on $X$ is generated by the base $\{g^{-1}(V) : V\in\mathbb{V}\}$, where $g=f\times f$. Now let $x\in X$. Since $Y$ is locally $\Ub$, there exists a $U\in\mathbb{V}$ such that $(Z,\mathbb{V}_Z)$ is a $\Ub$ subspace of $Y$, where $Z=U[f(x)]$. Let $W=f^{-1}(Z)$ and define $h:W \to Z$ by $h(u)=f(u)$. By Theorem~\ref{L4}, $(W,\mathbb S)$ is $\Ub$, where $\mathbb S=h^{-1}(\mathbb{V}_Z)$. Let $\tilde{g}=h\times h$. Now observe that $\mathbb{B}=\{\tilde{g}^{-1}(V) : V\in\mathbb{V}_Z\}$ is a base for the uniformity $\mathbb S$ on $W$ and $\mathbb{B}^\prime=\{(W\times W)\cap g^{-1}(V) : V\in\mathbb{V}\}$ is a base for the uniformity $\mathbb S^\prime = {f^{-1}(\mathbb{V})}_{W}$ on $W$. It is easy to verify that $\mathbb{B}=\mathbb{B}^\prime$ and hence $\mathbb S=\mathbb S^\prime$. Thus, $(W,\mathbb S)$ is a $\Ub$ subspace of $X$. Clearly, $g^{-1}(U)\in f^{-1}(\mathbb{V})$ and $g^{-1}(U)[x]=W$. Consequently, $g^{-1}(U)[x]$ is a $\Ub$ subspace of $X$ and the proof is now complete.
\end{proof}

We now observe that locally $\Ub$ property remains invariant under certain mappings. First we recall the following definitions from \cite{Bi-quotient}.
A surjective continuous mapping $f:X\to Y$ is said to be weakly perfect if $f$ is closed and $f^{-1}(y)$ is Lindel\"{o}f for each $y\in Y$.
Also a surjective continuous mapping $f:X\to Y$ is said to be bi-quotient if whenever $y\in Y$ and $\mathcal{U}$ is a cover of $f^{-1}(y)$ by open sets in $X$, then finitely many $f(U)$ with $U\in\mathcal{U}$ cover some open set containing $y$ in $Y$.
It is immediate that surjective continuous open (and also perfect) mappings are bi-quotient.

\begin{Th}
\hfill
\begin{enumerate}[wide=0pt, label={\upshape(\arabic*)},ref={\theTh(\arabic*)},leftmargin=*]
  \item 
  If $f:(X,\mathbb{U})\to(Y,\mathbb{V})$ is a uniformly continuous, bi-quotient mapping from a locally $\Ub$ space $X$ onto $Y$, then $Y$ is also  locally $\Ub$.
  \item 
  If $f:(X,\mathbb{U})\to(Y,\mathbb{V})$ is a uniformly continuous, weakly perfect mapping from a locally $\Ub$ space $X$ onto  $Y$, then $Y$ is also locally $\Ub$.
\end{enumerate}
\end{Th}
\begin{proof}
$(1)$. For each $x\in X$ choose $U_x\in\mathbb{U}$ such that $U_x[x]$ is  $\Ub$. Consider the open cover  $\{\Int U_x[x] : x\in X\}$ of $X$. Let $y\in Y$. Since $f$ is a bi-quotient mapping, there is a finite set $\{\Int U_{x_i}[x_i] : 1\leq i\leq k\}\subseteq\{\Int U_x[x] : x\in X\}$ and an open set $V$ in $Y$ such that $y\in V\subseteq\cup_{i=1}^kf(\Int U_{x_i}[x_i])$; that is $V\subseteq\cup_{i=1}^kf(U_{x_i}[x_i])$. By Theorem~\ref{LU404} and Theorem~\ref{TU4}, $\cup_{i=1}^kf(U_{x_i}[x_i])$ is a $\Ub$ subspace of $Y$. Next we can choose $W\in\mathbb{V}$ with $W[y]\subseteq V$ since $V$ is a neighbourhood of $y$ in  $Y$. Consequently, $W[y]$ is the required $\Ub$ subspace of $Y$.\\
$(2)$. Let $y\in Y$ and say $A=f^{-1}(y)$. For each $x\in A$ choose $U_x\in\mathbb{U}$ such that $U_x[x]$ is $\Ub$. Consider the cover $\{\Int U_x[x] : x\in A\}$ of $A$ by open sets in $X$. Since $A$ is Lindel\"{o}f, there is a countable collection $\{\Int U_{x_n}[x_n] : n\in\mathbb{N}\}$ that covers $A$; that is $A\subseteq\cup_{n\in\mathbb{N}}U_{x_n}[x_n]$. Observe that $y\in Y\setminus f(X\setminus\cup_{n\in\mathbb{N}}\Int U_{x_n}[x_n])\subseteq f(\cup_{n\in\mathbb{N}}U_{x_n}[x_n])$. By Theorem~\ref{TU4} and Theorem~\ref{LU404}, $f(\cup_{n\in\mathbb{N}}U_{x_n}[x_n])$ is a $\Ub$ subspace of $Y$. Since $f$ is closed, $Y\setminus f(X\setminus\cup_{n\in\mathbb{N}}\Int U_{x_n}[x_n])$ is an open set in $Y$ containing $y$. Thus, we obtain a $V\in\mathbb{V}$ such that $V[y]\subseteq Y\setminus f(X\setminus\cup_{n\in\mathbb{N}}\Int U_{x_n}[x_n])$. Clearly, $V[y]$ is a $\Ub$ subspace of $Y$, which completes the proof.
\end{proof}

\begin{Cor}
\label{CU1}
If $f:(X,\mathbb{U})\to(Y,\mathbb{V})$ is a uniformly continuous, perfect (or a uniformly continuous, open) mapping from a locally $\Ub$ space $X$ onto $Y$, then $Y$ is also  locally $\Ub$.
\end{Cor}

\begin{Th}
\label{TU5}
Let $\{X_\alpha: \alpha\in\Lambda\}$ be a family of open subspaces of a uniform space $(X,\mathbb{U})$ satisfying $X=\cup_{\alpha\in\Lambda}X_\alpha$. Then $X$ is locally $\Ub$ if and only if each $X_\alpha$ is locally $\Ub$.
\end{Th}

\begin{proof}
 We only need to prove sufficiency. Choose $x\in X$ and say $x\in X_{\beta}$. Since $(X_{\beta},\mathbb{U}_{X_{\beta}})$ is locally $\Ub$, choose $V\in\mathbb{U}_{X_{\beta}}$ such that $V[x]$ is a $\Ub$ subspace of $X_{\beta}$. Since $x\in \Int_{X_{\beta}}V[x]$ is open in $X$, there exists a $U\in\mathbb{U}$ such that $U[x]\subseteq \Int_{X_{\beta}}V[x]$. It follows that $U[x]$ is  $\Ub$ and therefore, $X$ is locally $\Ub$.
\end{proof}

\begin{Cor}
\label{TU3}
Let $\{(X_\alpha,\mathbb{U}_\alpha) : \alpha\in\Lambda\}$ be a family of uniform spaces. The sum $\oplus_{\alpha\in\Lambda}X_\alpha$ is locally $\Ub$ if and only if each $X_\alpha$ is locally $\Ub$.
\end{Cor}

\begin{Th}
 Let $\{X_\alpha: \alpha\in\Lambda\}$ be a locally finite family of closed subspaces of a uniform space $(X,\mathbb{U})$ satisfying $X=\cup_{\alpha\in\Lambda}X_\alpha$. Then $X$ is locally $\Ub$ if and only if each $X_\alpha$ is locally $\Ub$.
\end{Th}
\begin{proof}
We only need to prove sufficiency. Let $\mathbb{V}$ be the uniformity on the sum $Y=\oplus_{\alpha\in\Lambda}X_\alpha$. Define $f:Y\to X$ by $f(x,\alpha)=x$. We claim that $f$ is a uniformly continuous, perfect mapping. Let $U\in\mathbb{U}$. For each $\alpha$ choose $U_\alpha\in\mathbb{U}_{X_\alpha}$ such that $\cup_{\alpha\in\Lambda}U_\alpha\subseteq U$. Also for each $\alpha$ define $V_\alpha=\{((x,\alpha),(y,\alpha)) : (x,y)\in U_\alpha\}$. Clearly, $V=\cup_{\alpha\in\Lambda} V_\alpha\in\mathbb{V}$. Let $(u,v)\in V$ and choose a $\beta$ so that $(u,v)\in V_\beta$. Thus, $u=(x,\beta)$ and $v=(y,\beta)$ for some $(x,y)\in U_\beta$ and hence $(f(u),f(v))\in U_\beta$, so that $(f(u),f(v))\in U$. It follows that $f$ is uniformly continuous.

To show that $f$ is perfect, for each $\alpha$ we define a continuous mapping $\varphi_\alpha:X_\alpha\to Y$ by $\varphi_\alpha(x)=(x,\alpha)$. Let $F$ be closed in $Y$. Since each $\varphi_\alpha^{-1}(F)$ is closed  and $X_\alpha$'s are locally finite in $X$, $f(F)=\cup_{\alpha\in\Lambda}\varphi_\alpha^{-1}(F)$ is also closed in $X$. Since $f^{-1}(x)$ is finite for every $x\in X$, it is compact. Thus, $f$ is a perfect mapping, so that by Corollary~\ref{CU1}, $X$ is locally $\Ub$.
\end{proof}

The next two results concern respectively the locally $\Hub$ and locally $\Meb$ property in product spaces.

\begin{Th}
Let $(X,\mathbb{U})$ and $(Y,\mathbb{V})$ be two uniform spaces. Then $(X\times Y,\mathbb{U}\times\mathbb{V})$ is locally $\Hu$-bounded if and only if both $(X,\mathbb{U})$ and $(Y,\mathbb{V})$ are locally $\Hu$-bounded.
\end{Th}
\begin{proof}
 By considering the projection mappings, necessity follows from Corollary~\ref{CU1}.

Conversely assume that both $X$ and $Y$ are locally $\Hu$-bounded. Let $(x,y)\in X\times Y$. Choose $U\in\mathbb{U}$ and $V\in\mathbb{V}$ such that $U[x]$ and $V[y]$ are $\Hu$-bounded subspaces of $X$ and $Y$ respectively. By Theorem~\ref{LU501}, $U[x]\times V[y]$ is a $\Hu$-bounded subspace of $X\times Y$. Clearly, $(x,y)\in \Int_X U[x]\times\Int_Y V[y]$ is an open set in $X\times Y$. Thus, there is a $W\in\mathbb{U}\times\mathbb{V}$ such that $W[(x,y)]\subseteq \Int_X U[x]\times\Int_Y V[y]$. Consequently, $W[(x,y)]$ is a $\Hu$-bounded subspace of $X\times Y$ and hence $X\times Y$ is locally $\Hu$-bounded.
\end{proof}
 It is to be noted that locally $\Rob$ property is not productive (see Example~\ref{EU7}). We are unable to verify whether the above result holds for locally $\Meb$ spaces. Still, we have the following observation that can be easily verified. Moreover, the same example also demonstrates that `locally $\Meb$' cannot be replaced by `locally $\Rob$' in the following result.
\begin{Th}
Let $(X,\mathbb{U})$ and $(Y,\mathbb{V})$ be two uniform spaces. If $X$ is locally $\Me$-bounded and $Y$ is locally precompact, then $(X\times Y,\mathbb{U}\times\mathbb{V})$ is locally $\Me$-bounded.
\end{Th}

We end this section with some observations on the Baire space.
Combining Theorem~\ref{TU8}, Theorem~\ref{TU9}, Theorem~\ref{TU10}, Corollary~\ref{CU1} and Proposition~\ref{PU3}, we obtain the following result.
\begin{Prop}
Let $(X,\mathbb{U})$ be a uniform space.
\begin{enumerate}[wide=0pt,label={\upshape(\arabic*)},ref={\theProp(\arabic*)},leftmargin=*]
  \item 
  If $X$ is  locally $\Meb$, then every open, uniformly continuous image of $X$ into $\mathbb{N}^\mathbb{N}$ is non-dominating.
  \item 
  If $X$ is locally $\Hub$, then every open, uniformly continuous image of $X$ into $\mathbb{N}^\mathbb{N}$ is bounded (with respect to $\leq^*$).
  \item 
  If $X$ is locally $\Rob$, then every open, uniformly continuous image of $X$ into $\mathbb{N}^\mathbb{N}$ can be guessed.
\end{enumerate}
\end{Prop}

\section{Examples}
We now present examples to illustrate the distinction between the behaviours of the local variations as introduced in this article.

Recall that a collection $\mathcal{A}$ of subsets of $\mathbb{N}$ is said to be an almost disjoint family if each $A\in\mathcal{A}$ is infinite and for every two distinct elements $B,C\in\mathcal{A}$, $|B\cap C|<\aleph_0$. Also $\mathcal{A}$ is said to be a maximal almost disjoint (in short, MAD) family if $\mathcal{A}$ is not contained in any larger almost disjoint family. For an almost disjoint family $\mathcal{A}$, let $\Psi(\mathcal{A})=\mathcal{A}\cup\mathbb{N}$ be the Isbell-Mr\'{o}wka space. It is well known that $\Psi(\mathcal{A})$ is a locally compact zero-dimensional Hausdorff space (and hence is a Tychonoff space) (see \cite{Gillman,Mrowka}).

Note that $\Ub$ implies locally $\Ub$. The following example shows that the class of locally $\Ub$ uniform spaces properly contains the class of $\Ub$ uniform spaces.
\begin{Ex}[{locally $\Ub\centernot\implies \Ub$}]
\label{EU3}
Let $\Psi(\mathcal{A})$ be the Isbell-Mr\'{o}wka space with $|\mathcal{A}|>\aleph_0$. Let $\mathbb{U}$ be the corresponding uniformity on $\Psi(\mathcal{A})$. Since $\mathcal{A}$ is an uncountable discrete subspace of $\Psi(\mathcal{A})$, it follows that $\Psi(\mathcal{A})$ is not pre-Lindel\"{o}f. Thus, $\Psi(\mathcal{A})$ cannot be $\Ub$. Since for each member of $\Psi(\mathcal{A})$ we can find a countable basic open set containing it, it follows that $\Psi(\mathcal{A})$ is locally $\Ub$ by Proposition~\ref{PU2}.
\end{Ex}

The preceding example can also be used to show the existence of locally pre-Lindel\"{o}f (resp. locally precompact) space which is not pre-Lindel\"{o}f (resp. precompact).

\begin{Rem}
\label{RN010}
 The Isbell-Mr\'{o}wka space $\Psi(\mathcal{A})$ is $\Ub$ if and only if $|\mathcal{A}|\leq\aleph_0$.
\end{Rem}

We now give an example of a locally $\Hu$-bounded (and hence locally $\Me$-bounded) space which is not locally $\Ro$-bounded. First we recall that a set $A\subseteq\mathbb{R}$ has strong measure zero if for every sequence $(\varepsilon_n)$ of positive reals there exists a sequence $(I_n)$ of intervals such that the length of the interval $I_n$ is less than $\varepsilon_n$ for all $n$ and $A\subseteq\cup_{n\in\mathbb{N}}I_n$.

\begin{Ex}[{locally $\Hu$ (or, $\Me$)-bounded$\centernot\implies$ locally $\Rob$}]
\label{EU8}
Consider $\mathbb{R}$ with the uniformity induced by the standard metric $d$. Clearly, $\mathbb{R}$ is locally precompact and hence is locally $\Hub$ and locally $\Meb$ as well. Now if possible suppose that $\mathbb{R}$ is locally $\Rob$. By Proposition~\ref{PU3}, $\mathbb{R}$ is $\Rob$. Let $(\varepsilon_n)$ be a sequence of positive real numbers. Apply $\Rob$ property of $\mathbb{R}$ to the sequence of $\frac{\varepsilon_n}{3}$-entourages to obtain a sequence $(x_n)$ of reals such that $\mathbb{R}=\cup_{n\in\mathbb{N}}B_d(x_n,\frac{\varepsilon_n}{3})$. This implies that $\mathbb{R}$ has strong measure zero, a contradiction. Thus, $\mathbb{R}$ fails to be locally $\Rob$.
\end{Ex}

The following is an example of a locally pre-Lindel\"{o}f space which is not locally $\Ub$.
\begin{Ex}[{locally pre-Lindel\"{o}f $\centernot\implies$ locally $\Ub$}]

Let $\mathbb{R}^\omega$ be the Tychonoff product of $\omega$-copies of $\mathbb{R}$. The topology of $\mathbb{R}^\omega$ is induced by the metric $d(x,y)=\sup
\limits_{i\in\mathbb{N}}\frac{1}{i}\min\{|x_i-y_i|,1\}$ on $\mathbb{R}^\omega$, where $x=(x_i)$ and $y=(y_i)$. Let $\mathbb{U}$ be the uniformity on $\mathbb{R}^\omega$ induced by $d$. We claim that $\mathbb{R}^\omega$ is not $\Me$-bounded. On the contrary, assume that $\mathbb{R}^\omega$ is $\Me$-bounded. For each $n\in\mathbb{N}$ set $U_n=\{(x,y)\in\mathbb{R}^\omega\times\mathbb{R}^\omega : x=(x_i),y=(y_i)\;\text{and}\;|x_n-y_n|<n\}$. For any $0<\varepsilon<1$, $U_n$ contains the entourage $U_{\frac{\varepsilon}{n}}=
\{(x,y)\in\mathbb{R}^\omega\times\mathbb{R}^\omega : d(x,y)<\frac{\varepsilon}{n}\}$, so that $U_n\in\mathbb{U}$ for each $n$. Apply $\Me$-bounded property of $\mathbb{R}^\omega$ to $(U_n)$ to obtain a sequence $(F_n)$ of finite subsets of $\mathbb{R}^\omega$ such that $\mathbb{R}^\omega=\cup_{n\in\mathbb{N}}U_n[F_n]$. Say $F_n=\{x^{(n,j)}=(x_i^{(n_,j)}) : 1\leq j\leq k_n\}$ for each $n$. Choose $x=(x_i)\in\mathbb{R}^\omega$ such that $x_n=n+\Sigma_{j=1}^{k_n}|x_n^{(n,j)}|$ for each $n$. Also choose $n_0\in\mathbb{N}$ such that $(x,x^{(n_0,j_0)})\in U_{n_0}$ for some $x^{(n_0,j_0)}\in F_{n_0}$. By the construction of $U_{n_0}$, we obtain $|x_{n_0}-x_{n_0}^{(n_0,j_0)}|<n_0$, which is a contradiction as $x_{n_0}=n_0+\Sigma_{j=1}^{k_{n_0}}|x_{n_0}^{(n_0,j)}|$. Thus, $\mathbb{R}^\omega$ is not $\Me$-bounded. Now apply Proposition~\ref{PU3} to conclude that $\mathbb{R}^\omega$ is not locally $\Me$-bounded (and hence it is not locally $\Ub$ by Figure~\ref{dig1}).
\end{Ex}

We now give an example of a locally $\Ub$ space which is not locally precompact.
\begin{Ex} [{locally $\Ub$ $\centernot\implies$ locally precompact}]
Let $X$ be the hedgehog metric space (see \cite{Engelking}) of spininess $\aleph_0$. Then $X$ is a complete $\Ub$ space. Observe that $X$ is not locally compact. Since every complete precompact space is compact, a similar observation in line of Theorem~\ref{P41} shows that a complete locally precompact space is locally compact. Therefore, $X$ is not locally precompact.
\end{Ex}

We now make a quick observation which reflects that Lemma~\ref{TU2} does not hold for locally $\Ub$ spaces.
\begin{Ex}
Let $\kappa>\aleph_0$. Consider the space $X$ with the hedgehog metric $\rho$ of spininess $\kappa$. Clearly, $X$ is complete, but not locally Lindel\"{o}f. Let $\mathbb{U}$ be the uniformity on $X$ induced by $\rho$. Since every complete locally pre-Lindel\"{o}f space is locally Lindel\"{o}f, we can say that $(X,\mathbb{U})$ is not locally pre-Lindel\"{o}f. It now follows that $(X,\mathbb{U})$ is not locally $\Ub$. But the uniform space $(X,\mathbb{D})$ with discrete uniformity $\mathbb{D}$ is locally $\Ub$.
\end{Ex}


Recall that the product of a $\Me$-bounded (resp. $\Hu$-bounded) space with a precompact space is again $\Me$-bounded (resp. $\Hu$-bounded), but if we replace `precompact'  by `locally precompact', then the product need not be $\Me$-bounded (resp. $\Hu$-bounded).

\begin{Ex}[{$\Me$ (resp. $\Hu$)-bounded $\times$ locally precompact $\centernot\implies$ $\Me$ (resp. $\Hu$)-bounded }]

Let $X$ be the Isbell-Mr\'{o}wka space $\Psi(\mathcal{A})$ as in Example~\ref{EU3}. Let $\mathbb{U}$ be the corresponding uniformity on $X$. By Example~\ref{EU3}, $X$ is locally precompact but not pre-Lindel\"{o}f. Let $Y=\mathbb{R}$ with the standard metric uniformity $\mathbb{V}$. Since $Y$ is $\sigma$-compact, $Y$ is $\Meb$ as well as $\Hub$. If possible suppose that $X\times Y$ is $\Me$-bounded. Using the projection mapping and then by applying Theorem~\ref{LU404}, we have $X$ is $\Meb$, which is a contradiction. Thus, $X\times Y$ is not $\Meb$ and hence it fails to be $\Hub$.
\end{Ex}

We now observe that locally $\Ro$-bounded property behaves somewhat differently from locally $\Me$-bounded and locally $\Hu$-bounded property.

\begin{Ex}[locally $\Ro$-bounded $\times$ locally precompact $\centernot\implies$ locally $\Ro$-bounded]
\label{EU7}
This example is about the product of a locally $\Ro$-bounded space and a locally precompact space which fails to be locally $\Ro$-bounded.
To prove this, consider a $\Ro$-bounded uniform space $X$ and $\mathbb R$ with the standard metric uniformity. Suppose if possible that $X\times\mathbb{R}$ is locally $\Ro$-bounded. By Corollary~\ref{CU1} and by means of projection mapping, it follows that $\mathbb{R}$ is locally $\Ro$-bounded, which is a contradiction (see Example~\ref{EU8}). Therefore, the product $X\times\mathbb{R}$ fails to be locally $\Ro$-bounded.
\end{Ex}

\begin{Rem}
  Product of a $\Ro$-bounded space and a precompact space need not be $\Ro$-bounded. The reason is as follows. If possible suppose that the assertion is true. Now consider Example~\ref{EU7}. Since $\mathbb R$ is a countable union of its precompact subspaces, our supposition together with Theorem~\ref{TU4} imply $X\times\mathbb{R}$ is $\Ro$-bounded. We then again arrive at a contradiction as in Example~\ref{EU7}. Thus, the product of a $\Ro$-bounded space and a precompact space need not be $\Ro$-bounded.
\end{Rem}

We now present an example of a $\Ub$ space for which Corollary~\ref{TU3} fails to hold.
\begin{Ex}
Consider the sum $\oplus_{\alpha\in\omega_1}X$ of $\omega_1$ copies of a $\Ub$ space $(X,\mathbb{U})$. The sum is uniformly isomorphic to  $(X\times\omega_1,\mathbb{U}\times\mathbb{D})$, where $\mathbb{D}$ is the discrete uniformity on $\omega_1$. By Theorem~\ref{LU404}, $X\times\omega_1$ is not $\Ub$ and consequently, $\oplus_{\alpha\in\omega_1}X$ is not $\Ub$.
\end{Ex}

\begin{Rem}
\mbox{}\hfill
\begin{enumerate}[wide=0pt,
label={\upshape(\arabic*)},
ref={\theRem(\arabic*)},leftmargin=*]
\item It is interesting to observe that Theorem~\ref{TU5} does not hold for $\Ub$ spaces. Consider the Isbell-Mr\'{o}wka space $\Psi(\mathcal{A})$ as in Example~\ref{EU3}. As noted in that example, each point of $\Psi(\mathcal{A})$ has a countable basic open set containing it. Thus, $\Psi(\mathcal{A})$, which is itself not $\Ub$, is a union of $\Ub$ open subspaces.
\item Also note that Theorems~\ref{TN01}, \ref{TN02} and \ref{TN03} can not be extended to locally pre-Lindel\"{o}f spaces.
Assume that $\omega_1<\min\{\mathfrak b, \cov(\mathcal M)\}$ (which implies $\omega_1<\mathfrak d$ too). Let $\Psi(\mathcal{A})$ be the Isbell-Mr\'{o}wka space with $|\mathcal{A}|=\omega_1$. By Example~\ref{EU3}, $\Psi(\mathcal{A})$ is locally pre-Lindel\"{o}f. However, by Remark~\ref{RN010}, $\Psi(\mathcal{A})$ is not $\Ub$.
\end{enumerate}
\end{Rem}

\noindent{\bf Acknowledgement:} The authors are thankful to the Referee for his/her several valuable suggestions which significantly improved the presentation of the paper.
\paragraph{}

\end{document}